\documentclass[11pt]{article}
\usepackage{amssymb,amsmath,mathrsfs,amsthm,amsxtra,indentfirst,mathtools}
\usepackage[colorlinks,linkcolor=blue]{hyperref}
\usepackage[numbers,sort&compress]{natbib}
\numberwithin{equation}{section}

\newtheorem{lemma}{Lemma}[section]
\newtheorem{theorem}{Theorem}[section]
\newtheorem{remark}{Remark}[section]

\newcommand\tbbint{{-\mkern -16mu\int}}

\newcommand\dbbint{{-\mkern -19mu\int}}

\newcommand\bbint{
{\mathchoice{\dbbint}{\tbbint}{\tbbint}{\tbbint}}
}

\renewcommand{\leq}{\leqslant}
\renewcommand{\geq}{\geqslant}

\def\R{\mathbb{R}}
\def\T{\mathbb{T}}
\def\D{\mathcal{D}}
\def\norm#1#2{\left\|#1\right\|_{#2}}
\def\set#1#2{\left\{#1 \,\middle|\, #2\right\}}

\DeclareMathOperator\dif{d\!}

\DeclareMathOperator{\curl}{curl}

\allowdisplaybreaks

\begin{document}

\title{
\bf Liouville--type Theorems for Steady MHD and Hall--MHD Equations in $\R^2 \times \T$}

\author{Wentao Hu, Zhengce Zhang\thanks{Corresponding author. Email: zhangzc@mail.xjtu.edu.cn}
\\
{\small \it School of Mathematics and Statistics, Xi'an Jiaotong University}
\\
{\small \it Xi'an, Shaanxi 710049, PR China}
}

\date{}

\maketitle

\begin{abstract}
    In this paper, we study the Liouville--type theorems for three--dimensional stationary incompressible MHD and Hall--MHD systems in a slab with periodic boundary condition. We show that, under the assumptions that $(u^\theta,b^\theta)$ or $(u^r,b^r)$ is axisymmetric, or $(ru^r,rb^r)$ is bounded, any smooth bounded solution to the MHD or Hall--MHD system with local Dirichlet integral growing as an arbitrary power function must be constant. This hugely improves the result of \cite[Theorem 1.2]{pan2021Liouville}, where the Dirichlet integral of $\mathbf{u}$ is assumed to be finite. Motivated by \cite[Bang--Gui--Wang--Xie, 2022, {\it arXiv:2205.13259}]{bang2022Liouvilletype}, our proof relies on establishing Saint--Venant's estimates associated with our problem, and the result in the current paper extends that for stationary Navier--Stokes equations shown by \cite{bang2022Liouvilletype} to MHD and Hall--MHD equations. To achieve this, more intricate estimates are needed to handle the terms involving $\mathbf{b}$ properly.

    \vspace{1em}


    \vspace{1em}

\end{abstract}

\section{Introduction}
In this paper, we consider the stationary magnetohydrodynamic (MHD) equations
\begin{equation}
    \label{eq: MHD}
    \begin{dcases}
    \mathbf{u} \cdot \nabla \mathbf{u} + \nabla p - \mathbf{b} \cdot \nabla \mathbf{b} - \Delta \mathbf{u} = 0, \\
    \mathbf{u} \cdot \nabla \mathbf{b} - \mathbf{b} \cdot \nabla \mathbf{u} - \Delta \mathbf{b} = 0, \\
    \nabla \cdot \mathbf{u} = \nabla \cdot \mathbf{b} = 0,
    \end{dcases}
    \quad
    \text{in} \ \Omega \subset \R^3
\end{equation}
and the stationary Hall--MHD equations
\begin{equation}
    \label{eq: Hall-MHD}
    \begin{dcases}
    \mathbf{u} \cdot \nabla \mathbf{u} + \nabla p - \mathbf{b} \cdot \nabla \mathbf{b} - \Delta \mathbf{u} = 0, \\
    \mathbf{u} \cdot \nabla \mathbf{b} - \mathbf{b} \cdot \nabla \mathbf{u} - \Delta \mathbf{b} = \nabla \times ((\nabla \times \mathbf{b}) \times \mathbf{b}), \\
    \nabla \cdot \mathbf{u} = \nabla \cdot \mathbf{b} = 0,
    \end{dcases}
    \quad
    \text{in} \ \Omega \subset \R^3,
\end{equation}
where the unknown vector fields $\mathbf{u}=(u^1,u^2,u^3)$, $\mathbf{b}=(b^1,b^2,b^3)$ and scalar function $p$ stand for the velocity field, the magnetic field, and the pressure, respectively. The system \eqref{eq: MHD} is used to depict the steady state of incompressible charged flows under the effect of magnetic field (e.g. plasma). When $\mathbf{b} \equiv 0$, the system \eqref{eq: MHD} reduces to the stationary Navier--Stokes system
\begin{equation}
    \label{eq: NS}
    \begin{dcases}
    \mathbf{u} \cdot \nabla \mathbf{u} + \nabla p - \Delta \mathbf{u} = 0, \\
    \nabla \cdot \mathbf{u} = 0,
    \end{dcases}
    \quad
    \text{in} \ \Omega \subset \R^3,
\end{equation}
as has been intensively studied during the past decades. The Hall--MHD system \eqref{eq: Hall-MHD}, compared with \eqref{eq: MHD}, takes into account the impact of the discrepancy between electron and ion velocities in Ohm's law, which may become significant in situations like magnetic reconnection. \eqref{eq: Hall-MHD} has many important applications in the study of solar flares, neutron stars, space plasmas, etc. For more explanation on the physical background of the Hall--MHD system, we refer readers to \cite{balbus2001Linear,forbes1991Magnetic,homann2005Bifurcation,lighthill1960Studies,shalybkov1997Hall,wardle2004Star} and the references therein. The first systematic study of the Hall effects was launched by Lighthill in \cite{lighthill1960Studies}, where he physically derived the Hall term $\nabla \times ((\nabla \times \mathbf{b}) \times \mathbf{b})$ in \eqref{eq: Hall-MHD}$_2$ for the first time. The first rigorous mathematical derivation of the Hall--MHD system, however, was due to Acheritogaray et al. in \cite{acheritogaray2011Kinetic}, where they used a set of scaling limits to get their result from either two--phase flow model or kinetic model. See also \cite{srinivasan2011Analytical}. In \cite{jang2012Derivation}, Jang and Masmoudi carried out a formal derivation of the Hall effect by using the asymptotic analysis, and in \cite{peng2022Derivation}, Peng et al. derived the Hall--MHD equations from the Navier--Stokes--Maxwell equations with generalized Ohm’s law in a mathematically rigorous way via the spectral analysis and energy methods.

For our purpose in this paper, it is more convenient to study \eqref{eq: MHD} and \eqref{eq: Hall-MHD} in a joint way. To this end, we consider the following unified system with a constant $\delta \in [0,1]$:
\begin{equation}
    \label{eq: MHD_Hall-MHD}
    \begin{dcases}
    \mathbf{u} \cdot \nabla \mathbf{u} + \nabla p - \mathbf{b} \cdot \nabla \mathbf{b} - \Delta \mathbf{u} = 0, \\
    \mathbf{u} \cdot \nabla \mathbf{b} - \mathbf{b} \cdot \nabla \mathbf{u} - \Delta \mathbf{b} = \delta \nabla \times ((\nabla \times \mathbf{b}) \times \mathbf{b}), \\
    \nabla \cdot \mathbf{u} = \nabla \cdot \mathbf{b} = 0,
    \end{dcases}
    \quad
    \text{in} \ \Omega \subset \R^3.
\end{equation}
Apparently, \eqref{eq: MHD_Hall-MHD} reduces to \eqref{eq: MHD} for $\delta=0$ and \eqref{eq: Hall-MHD} for $\delta=1$.

When we consider the cylindrical coordinate $(r,\theta,z)$, which is defined as
\begin{equation*}
    x_1=r\cos\theta, \quad x_2=r\sin\theta, \quad x_3=z,
\end{equation*}
and denote the components of $\mathbf{u},\mathbf{b}$ as $\mathbf{u}=u^r e_r+u^\theta e_\theta+u^z e_z$, $\mathbf{b}=b^r e_r+b^\theta e_\theta+b^z e_z$, where $u^r,u^\theta$ and $u^z$ are called radial, swirl, and axial velocity, respectively, with
\begin{equation*}
    e_r=(\cos \theta,\sin \theta,0), \quad e_\theta=(-\sin \theta,\cos \theta,0), \quad e_z=(0,0,1),
\end{equation*}
the momentum equation \eqref{eq: MHD_Hall-MHD}$_1$ reads
\begin{equation}
    \label{eq: cylindrical}
    \begin{dcases}
    \mathbf{u} \cdot \nabla u^r-\frac{(u^\theta)^2}{r} + \frac{2}{r^2} \partial_\theta u^\theta+\partial_r p = \mathbf{b} \cdot \nabla b^r - \frac{(b^\theta)^2}{r} + \bigg( \Delta-\frac{1}{r^2} \bigg) u^r, \\
    \mathbf{u} \cdot \nabla u^\theta+\frac{u^\theta u^r}{r}-\frac{2}{r^2} \partial_\theta u^r+\frac{1}{r} \partial_\theta p = \mathbf{b} \cdot \nabla b^\theta + \frac{b^r b^\theta}{r} + \bigg( \Delta-\frac{1}{r^2} \bigg) u^\theta, \\
    \mathbf{u} \cdot \nabla u^z + \partial_z p = \mathbf{b} \cdot \nabla b^z + \Delta u^z,
    \end{dcases}
\end{equation}
and the divergence free condition \eqref{eq: MHD_Hall-MHD}$_3$ reads
\begin{equation}
    \label{eq: divu=divb=0}
    \partial_r(ru^r)+\partial_\theta u^\theta+\partial_z(ru^z) = \partial_r(rb^r)+\partial_\theta b^\theta+\partial_z(rb^z) = 0,
\end{equation}
which is helpful when some or all of the components of $\mathbf{u}$ or $\mathbf{b}$ under the cylindrical coordinate system are assumed to be axisymmetric, that is, independent of $\theta$.

We are interested in the Liouville--type theorems for \eqref{eq: MHD} and \eqref{eq: Hall-MHD}, which have attracted a lot of attention due to their key role in characterizing the asymptotic behaviors of flows near potential singularities. We refer readers to the expository article \cite{seregin2018liouville} for more about the connection between the classification of singularities and Liouville--type theorems. In retrospect, the existence of solutions to the corresponding Navier--Stokes system \eqref{eq: NS} in an exterior domain $\Omega$ satisfying
\begin{equation}
    \label{eq: DuinL2}
    \int_\Omega |\nabla \mathbf{u}|^2 \dif x < \infty
\end{equation}
was studied by Leray in his seminal paper \cite{Leray1933Etude}. A solution to \eqref{eq: NS} in an exterior domain or the whole space is called a D--solution if it satisfies \eqref{eq: DuinL2}. A longstanding open problem is whether a D--solution in the whole space must be identically 0 when it tends to zero at far field. In the two--dimensional case, the answer is positive, see \cite[Gilbarg and Weinberger, 1978]{gilbarg1978asymptotic}. Whereas in the three--dimensional case, the problem has only been partially solved. For example, in \cite{galdi2011introduction} Galdi obtained the vanishing result in the case where $\Omega=\R^3$ and $\mathbf{u} \in L^{9/2}(\R^3)$, which was later improved by a log factor in \cite{chae2016Liouville} by Chae and Wolf. Also, in \cite{chae2014liouville}, Chae proved that the associated problem admits only trivial solution if $\Delta \mathbf{u} \in L^{6/5}(\R^3)$ is assumed. This condition is stronger than \eqref{eq: DuinL2}, but has the same scaling. In \cite{seregin2016Liouville-Nonlinearity} Seregin proved the vanishing of $\mathbf{u}$ under the condition $\Omega=\R^3$ and $\mathbf{u} \in L^6(\R^3) \cap BMO^{-1}$, which was different from the one in \cite{galdi2011introduction} as functions belonging to $L^6(\R^3) \cap BMO^{-1}$ but not $L^{9/2}(\R^3)$ exist. The restriction that $\mathbf{u} \in L^6(\R^3)$ was removed later by Seregin in \cite{seregin2018Remarks}. We refer readers to \cite{seregin2016Liouville,chae2019Liouville} for other generalizations of Seregin's result. In recent years, Chamorro, Jarr\'in and Lemari\'e-Rieusset extended Galdi's result to the case where $\mathbf{u} \in L_{loc}^2(\R^3) \cap L^p(\R^3)$ ($3 \leq p \leq 9/2$) in \cite{chamorro2021Liouville}, and the gap $p \in [2,3]$ was filled by Yuan and Xiao in \cite{yuan2020Liouvilletype}. Apart from the ``integral-type'' conditions, the nonexistence of nontrivial solutions to Navier--Stokes equations under decay assumptions is also of interest to many people. For instance, in \cite{kozono2017Remark} Kozono et al. showed that a D--solution $\mathbf{u} \equiv 0$ if the vorticity $\mathbf{w}=\curl \mathbf{u}$ decays like $o(|x|^{-5/3})$.

As for the axisymmetric case, Koch et al., among other results, proved in their prominent paper \cite{koch2009liouville} that any bounded axisymmetric solution vanishes if it decays like $|x'|^{-1}$ (here and below $x'$ denotes $(x_1,x_2)$), which was improved later by Wang in \cite{wang2019Remarks} and Zhao in \cite{zhao2019Liouville} independently to the case where $|\mathbf{u}| \leq C |x'|^{-\alpha}$ or $|\curl \mathbf{u}| \leq C |x'|^{-1-\alpha}$ for any axisymmetric D--solution $\mathbf{u}$ and all $\alpha>2/3$. Also, Koch et al. proved in \cite{koch2009liouville} the Liouville--type result for axisymmetric bounded solutions with no swirl (i.e. $u^\theta=0$). Chae and Weng proved in \cite{chae-weng2016Liouville} the vanishing of axisymmetric D--solution provided $|x'| u^r \geq -1$ and $\lim_{|x|\to\infty} \mathbf{u}=0$.

There are also numerous analogous results for three--dimensional stationary MHD and Hall--MHD equations. Chae et al., among other results, established the vanishing result for \eqref{eq: Hall-MHD} in \cite{chae2014wellposedness}, where he showed that if $\mathbf{u},\mathbf{b} \in L^\infty(\R^3) \cap L^{9/2}(\R^3)$, and $\nabla \mathbf{u},\nabla \mathbf{b} \in L^2(\R^3)$, then $\mathbf{u} = \mathbf{b} \equiv 0$. The boudedeness assumption on $\mathbf{u},\mathbf{b}$ was removed by Zhang et al. in \cite{zhang2015Remarks}. Later, in \cite{chae-weng2016Liouville}, Chae and Weng proved that any solution $(\mathbf{u},\mathbf{b})$ to \eqref{eq: MHD} or \eqref{eq: Hall-MHD} that approaches to $0$ at the far field is identically zero if $\mathbf{u} \in L^3(\R^3)$ and $\nabla \mathbf{u},\nabla \mathbf{b} \in L^2(\R^3)$, and Yuan and Xiao improved this result in \cite{yuan2020Liouvilletype} by removing the condition $\nabla \mathbf{u} \in L^2(\R^3)$. Indeed, it is proved in \cite{yuan2020Liouvilletype} that the solution $(\mathbf{u},\mathbf{b})$ to \eqref{eq: Hall-MHD} vanishes if $\mathbf{u},\mathbf{b} \in L^p(\R^3)$ ($2 \leq p \leq 9/2$) and $\nabla \mathbf{b} \in L^2(\R^3)$, or if $\mathbf{u} \in L^p(\R^3)$ ($2 \leq p \leq 9/2$) and $\mathbf{b} \in L^q(\R^3)$ ($4 \leq q \leq 9/2$). Also, the vanishing of the solution $(\mathbf{u},\mathbf{b})$ to \eqref{eq: MHD} was obtained in \cite{yuan2020Liouvilletype} under the condition that $\mathbf{u},\mathbf{b} \in L^p(\R^3)$ ($2 \leq p \leq 9/2$). We remind the readers that the results of Yuan and Xiao for MHD and Hall--MHD equations also generalized the result in \cite{chae2014wellposedness}, which is mentioned at the beginning of this paragraph, and is analogous to Yuan and Xiao's result for Navier--Stokes equations, as has been listed in the foregoing text. Moreover, in \cite{schulz2019Liouville} Schulz generalized Seregin's result for Navier--Stokes equations in \cite{seregin2016Liouville-Nonlinearity} to MHD equations \eqref{eq: MHD}, proving that if $\mathbf{u},\mathbf{b} \in L^p(\R^3) \cap BMO^{-1}$ ($2<p \leq 6$), then $\mathbf{u}=\mathbf{b} \equiv 0$. The assumption that $\mathbf{u},\mathbf{b} \in L^p(\R^3)$ was later removed by Chae in \cite{chae2021Liouville-MHD}. Actually, more general result was obtained in \cite{chae2021Liouville-MHD}, that is, if there exist $\Phi,\Psi \in C^\infty(\R^3;\R^{3 \times 3})$ such that $\nabla \cdot \Phi=\mathbf{u}$, $\nabla \cdot \Psi=\mathbf{b}$, and
\begin{equation*}
    \bbint_{B_R} |\Phi-\Phi_{B_R}|^6 \dif x + \bbint_{B_R} |\Psi-\Psi_{B_R}|^6 \dif x \leq CR, \quad \forall R>1,
\end{equation*}
then $\mathbf{u} = \mathbf{b} \equiv 0$. Here both $\bbint_{B_R} f \dif x$ and $f_{B_R}$ represent the mean value of the function $f$ over the ball $B_R$. This can also be regarded as a generalization of \cite{chae2019Liouville} for Navier--Stokes equations to MHD equations.

Compared to the whole $\R^3$ case, Liouville--type results can usually be established under more relaxed conditions if we restrict the domain to be a slab $\R^2 \times [0,1]$ with no--slip boundary condition
\begin{equation*}
    \mathbf{u},\mathbf{b}=0 \quad \text{at} \ x_3=0 \ \text{or} \ 1,
\end{equation*}
or periodic boundary condition in $x_3$ (i.e. $\mathbf{u},\mathbf{b}$ are regarded to be periodic in $x_3$, and in this case, we may denote the domain by $\R^2 \times \T$). In \cite[Theorem 1.1]{carrillo2020Decay-ARMA} Carrillo et al. showed that any bounded solution to the stationary Navier--Stokes equations in the slab $\Omega=\R^2 \times [0,1]$ subject to the no--slip boundary condition is identically zero if \eqref{eq: DuinL2} is satisfied. One of the key differences from the whole space case is that the no--slip boundary problem over a slab assures the validity of Poincar\'e's inequality, so does the periodic boundary case. In \cite{carrillo2020Decay-JFA}, Carrillo et al. studied the axisymmetric stationary Navier--Stokes equations over $\Omega=\R^2 \times \T$, and proved that if $u^\theta,u^z$ have zero mean value over $\T$,
\begin{equation*}
    \int_{B_{2R \backslash R} \times \T} |\nabla \mathbf{u}|^2 \dif x \leq C, \quad \forall R>1,
\end{equation*}
where $B_{2R \backslash R}=\set{(x_1,x_2)}{R^2 \leq x_1^2+x_2^2 \leq 4R^2} \subset \R^2$, then $\mathbf{u} \equiv 0$ provided it tends to zero at the far field. This result was improved in \cite[Theorem 1.3]{carrillo2020Decay-ARMA} by removing the zero mean value condition of $u^\theta,u^z$, at the price that \eqref{eq: DuinL2} is required to be satisfied. The methods of Carrillo et al. are based on estimating the decay rate of the Green's function over $\R^2 \times \T$, and can be applied to the case where the local Dirichlet integral $\mathcal{E}(R)=\int_0^1 \int_{|x'| \leq R} |\nabla \mathbf{u}|^2 \dif x$ admits some growth with respect to $R$. Indeed, it was shown in Carrillo's PhD thesis \cite{carrilloPHD} that any axisymmetric solution $\mathbf{u}$ to \eqref{eq: NS} over $\Omega=\R^2 \times \T$ that tends to $0$ at the far field vanishes if $u^\theta,u^z$ have zero mean value over $\T$ and $\mathcal{E}(R)$ grows slower than $R^{1/5}$. There are also analogous results for stationary MHD equations in a slab. For instance, in \cite{pan2021Liouville} Pan proved, among other results, that if $u^\theta,b^\theta$ are axisymmetric and $\nabla \mathbf{u} \in L^2(\R^2 \times \T)$, then the solution $(\mathbf{u},\mathbf{b})$ to \eqref{eq: MHD} in $\R^2 \times \T$, with $\mathbf{b}$ tending to zero at the far field, is trivial. This actually generalized the result in \cite[Theorem 1.3]{carrillo2020Decay-ARMA} by relaxing the axisymmetry assumption on $\mathbf{u}$.

In \cite[Theorem 1.1]{bang2022Liouvilletype}, Bang et al. substantially improved the results of Carrillo et al., proving that an axisymmetic solution $\mathbf{u}$ to the stationary Navier--Stokes equations with no--slip boundary condition over the slab $\R^2 \times [0,1]$ is identically zero if the local Dirichlet integral $\mathcal{E}(R)$ grows at a rate slower than $R^4$ (which is equivalent to $\varliminf_{R \to \infty} R^{-4} \mathcal{E}(R)=0$), or $\sup_z |\mathbf{u}(R,z)|$ is sublinear with respect to $R$ for large $R$. When the assumption of axisymmetry is dropped, the vanishing of $\mathbf{u}$ still holds in case that there exists a number $\alpha \in (0,1)$ such that $\mathcal{E}(R)$ grows slower than $R^\alpha$, or $\sup_{\theta,z} |\mathbf{u}(R,\theta,z)|$ grows slower than $R^\beta$ and $\sup_{\theta,z} |u^r(R,\theta,z)|$ grows no faster than $R^{-1+(\alpha/2-\beta)}$ for some $\beta \in [0,\alpha/2]$, as is shown in \cite[Theorem 1.2]{bang2022Liouvilletype}. For the periodic boundary case, it was proved in \cite[Theorem 1.4]{bang2022Liouvilletype} that any bounded solution to the stationary Navier--Stokes equations over $\R^2 \times \T$ takes the form of $(0,0,c)$, where $c$ is a constant, if $u^\theta$ or $u^r$ is axisymmetric, or $ru^r \to 0$ as $r \to \infty$.

The key idea of \cite{bang2022Liouvilletype} is to establish various differential inequalities for the local Dirichlet integral of a nontrivial solution and from that to estimate the growth rate of the local Dirichlet integral. This idea dates back to the so-called Saint--Venant's principle, which was first used to establish the decay rate estimates for equations of elasticity in \cite{toupin1965SaintVenant,knowles1966SaintVenant}, and was generalized to estimate the growth of local Dirichlet integral for solutions to linear elliptic equations in \cite{oleinik1977Boundary}. It was also generalized to deal with Navier--Stokes equations in a pipe by \cite{horgan1978Spatial} and \cite{ladyzhenskaya1980Determination}, and we refer readers to \cite[Chapter XIII]{galdi2011introduction} for further information.

Inspired by \cite{bang2022Liouvilletype}, in this paper we would like to establish several Liouville--type results for stationary MHD and Hall--MHD equations over a slab with periodic boundary condition by deriving Saint--Venant's estimates associated to our problem. As will be seen in Theorem \ref{thm: Liouville}, the results of this paper hold under fairly loose conditions in the sense that the local Dirichlet integral admits arbitrary polynomial growth (and so do the gradients of the solutions, as a consequence), which is hard to be violated in the real physical world. To the best knowledge of us, existing works on the Liouville--type theorems for MHD and Hall--MHD equations have more or less made relatively tough assumptions on the local Dirichlet integral, or other types of assumptions that are equally strict to some extent. One of the results in this direction is due to \cite[Theorem 1.2]{pan2021Liouville}, which shows the Liouville--type theorem for stationary MHD equations in $\Omega=\R^2 \times \T$ with \eqref{eq: DuinL2} being assumed, apart from the conditions that $u^\theta,b^\theta$ are axisymmetric and $\mathbf{b}$ tends to $0$ at the far field. We will drop the condition of \eqref{eq: DuinL2} in Theorem \ref{thm: Liouville} (i). Also, the results in the current paper extend that of \cite[Theorem 1.4]{bang2022Liouvilletype} to stationary MHD and Hall--MHD systems, where the Lorentz force is taken into account, and the equation of magnetic field, for the latter case, is attached with the Hall term $\nabla \times ((\nabla \times \mathbf{b}) \times \mathbf{b})$ which makes the equation quasilinear. This increases the complexity of the estimation, and more careful observations are needed to deal with the terms involving $\mathbf{b}$ properly.

The main theorem of this paper is as follows.
\begin{theorem}
\label{thm: Liouville}
Let $(\mathbf{u},\mathbf{b})$ be a smooth bounded solution to \eqref{eq: MHD_Hall-MHD} in $\R^2 \times \T$ (by $\R^2 \times \T$ we mean a slab $\R^2 \times [0,1]$ with $u,b$ being periodic in $x_3$, as is mentioned above). Suppose
\begin{equation}
    \label{eq: |Du|2+|Db|2<=Ralpha}
    \int_0^1 \int_{|x'| \leq R} |\nabla \mathbf{u}|^2+|\nabla \mathbf{b}|^2 \dif x \leq C R^\alpha,
\end{equation}
where $\alpha$ is an arbitrarily selected positive number, then $\mathbf{u}$ and $\mathbf{b}$ are constant vectors provided one of the following conditions holds:

(i) $u^\theta,b^\theta$ are axisymmetric;

(ii) $u^r,b^r$ are axisymmetric;

(iii) $ru^r,rb^r$ are bounded.

Furthermore, in cases (i) and (ii), the constant vectors must take the form of $\mathbf{u}=(0,0,c_1)$, $\mathbf{b}=(0,0,c_2)$, where $c_1$ and $c_2$ are constant numbers.
\end{theorem}
\begin{remark}
Case (i) of Theorem \ref{thm: Liouville} improves the result of \cite[Theorem 1.2]{pan2021Liouville} where $\nabla \mathbf{u} \in L^2(\R^2 \times \T)$ and $\lim_{|x'| \to \infty} \mathbf{b}=0$ are assumed besides the condition that $u^\theta$ and $b^\theta$ are axisymmetric to obtain the Liouville theorem.
\end{remark}

In the rest of this paper we denote $B_R=\set{(x_1,x_2)}{x_1^2+x_2^2 \leq R^2} \subset \R^2$, $\widetilde{D}_R=[R-1,R] \times [0,1]$, $\mathcal{D}_R=[R-1,R] \times [0,2\pi] \times [0,1]$, $O_R=(B_R \backslash B_{R-1}) \times \T$, and $\nabla=(\partial_r,\frac{1}{r} \partial_\theta,\partial_z)$, $\Delta=\partial_r^2+\frac{1}{r}\partial_r+\frac{1}{r^2}\partial_\theta^2+\partial_z^2$, $\widetilde{\nabla}=(\partial_r,\partial_z)$, $\overline{\nabla}=(\partial_r,\partial_\theta,\partial_z)$.
\section{Preliminary}
In this section we introduce some lemmas which are indispensable for our proof. The first lemma is from \cite{bang2022Liouvilletype} (whereas the general form is due to \cite{bogovskii1979Solution}), which is helpful in estimating the term involving $p$.
\begin{lemma}
\label{thm: Bogovskii}
For a general domain $\Omega$, denote
\begin{equation*}
    L_\sigma^2(\Omega):=\set{f \in L^2(\Omega)}{\int_{\Omega} f \dif x=0}.
\end{equation*}

(i) There exists a linear operator $\Phi$ that maps every $f \in L_\sigma^2(\widetilde{D}_R)$ into a vector field $V=\Phi f \in H_0^1(\widetilde{D}_R;\R^2)$, such that
\begin{equation*}
    \partial_r V^r+\partial_z V^z=f \quad \text{in} \ \widetilde{D}_R, \quad \text{and} \ \|\widetilde{\nabla} V\|_{L^2(\widetilde{D}_R)} \leq C \norm{f}{L^2(\widetilde{D}_R)}.
\end{equation*}

(ii) There exists a linear operator $\Phi$ that maps every $f \in L_\sigma^2(\D_R)$ into a vector field $V=\Phi f \in H_0^1(\D_R;\R^3)$, such that
\begin{equation*}
    \partial_r V^r+\partial_\theta V^\theta+\partial_z V^z=f \quad \text{in} \ \D_R, \quad \text{and} \ \|\overline{\nabla} V\|_{L^2(\D_R)} \leq C \norm{f}{L^2(\D_R)}.
\end{equation*}

In both cases $C$ represents a positive constant independent of $R$, and $\Phi$ is termed a Bogovski\u{i} map.
\end{lemma}
The second lemma is about the Poincar\'e's inequalities that fit in with $\mathbf{u},\mathbf{b}$.
\begin{lemma}
\label{thm: Poincare}
The following inequalities hold if either $(u^\theta,b^\theta)$ or $(u^r,b^r)$ is axisymmetric:
\begin{equation}
    \label{eq: Poincare}
    \norm{u^r}{L^2(O_R)} \leq C \norm{\partial_z u^r}{L^2(O_R)}, \quad \norm{b^r}{L^2(O_R)} \leq C \norm{\partial_z b^r}{L^2(O_R)}.
\end{equation}
\end{lemma}
\begin{proof}
Integrating \eqref{eq: divu=divb=0} with respect to $z$ over $[0,1]$, we obtain
\begin{equation}
    \label{eq: int(div=0)dz}
    \int_0^1 \partial_r(ru^r)+\partial_\theta u^\theta \dif z = \int_0^1 \partial_r(rb^r)+\partial_\theta b^\theta \dif z = 0.
\end{equation}
When $u^\theta,b^\theta$ are axisymmetric, i.e. $\partial_\theta u^\theta=\partial_\theta b^\theta=0$, the above equation, after being integrated with respect to $r$, gives
\begin{equation}
    \label{eq: urbrmean=0}
    \int_0^1 u^r \dif z = \int_0^1 b^r \dif z = 0.
\end{equation}
When $u^r$ and $b^r$ are axisymmetric, integrating \eqref{eq: int(div=0)dz} with respect to $\theta$ over $[0,2\pi]$ gives
\begin{equation*}
    2\pi \int_0^1 \partial_r(ru^r) \dif z = 2\pi \int_0^1 \partial_r(rb^r) \dif z = 0,
\end{equation*}
which also brings us to \eqref{eq: urbrmean=0} after being integrated with respect to $r$. Therefore we arrive at \eqref{eq: Poincare} by using the usual Poincar\'e's inequality
\begin{equation*}
    \int_0^1 \bigg| u^r-\bbint_0^1 u^r \dif z \bigg|^2 \dif z \leq C \int_0^1 |\partial_z u^r|^2 \dif z, \quad \int_0^1 \bigg| b^r-\bbint_0^1 b^r \dif z \bigg|^2 \dif z \leq C \int_0^1 |\partial_z b^r|^2 \dif z.
\end{equation*}
\end{proof}

\section{Proof of Theorem \ref{thm: Liouville}}
\label{sec: proof}
For $R>1$ large enough, define the axisymmetric cut-off function
\begin{equation}
    \varphi_R(r)=
    \begin{dcases}
    1, \quad 0 \leq r<R-1, \\
    R-r, \quad R-1 \leq r \leq R, \\
    0, \quad r>R.
    \end{dcases}
\end{equation}
Test \eqref{eq: MHD_Hall-MHD}$_1$ and \eqref{eq: MHD_Hall-MHD}$_2$ by $\varphi_R \mathbf{u}$ and $\varphi_R \mathbf{b}$ respectively and integrate them by parts over $\R^2 \times \T$ and then add them up, noting that\footnote{Here we've used the formula $\int_\Omega \mathbf{a} \cdot (\nabla \times \mathbf{b}) \dif x - \int_\Omega \mathbf{b} \cdot (\nabla \times \mathbf{a}) \dif x = \int_{\partial \Omega} (\mathbf{b} \times \mathbf{a}) \cdot \mathbf{n} \dif S$ for general 3-dimensional vector fields $\mathbf{a},\mathbf{b}$ and domain $\Omega \subset \R^3$.}
\begin{equation*}
    \begin{split}
    \int_{\R^2 \times \T} \nabla \times ((\nabla \times \mathbf{b}) \times \mathbf{b}) \cdot \mathbf{b} \varphi_R \dif x =& \int_{B_R \times \T} ((\nabla \times \mathbf{b}) \times \mathbf{b}) \cdot (\nabla \times (\varphi_R \mathbf{b})) \dif x \\
    =& \int_{B_R \times \T} ((\nabla \times \mathbf{b}) \times \mathbf{b}) \cdot (\varphi_R \nabla \times \mathbf{b} + \nabla \varphi_R \times \mathbf{b}) \dif x \\
    =& \int_{O_R} ((\nabla \times \mathbf{b}) \times \mathbf{b}) \cdot (\nabla \varphi_R \times \mathbf{b}) \dif x,
    \end{split}
\end{equation*}
therefore we obtain
\begin{align}
    &\int_{\R^2 \times \T} (|\nabla \mathbf{u}|^2+|\nabla \mathbf{b}|^2) \varphi_R \dif x \notag \\
    =& -\int_{O_R} \partial_r \varphi_R \partial_r \mathbf{u} \cdot \mathbf{u} \dif x + \frac{1}{2} \int_{O_R} (|\mathbf{u}|^2+|\mathbf{b}|^2) u^r \partial_r \varphi_R \dif x \notag \\
    &+ \int_{O_R} (p-p_R) u^r \partial_r \varphi_R \dif x - \int_{O_R} (\mathbf{u} \cdot \mathbf{b}) b^r \partial_r \varphi_R \dif x \notag \\
    &- \int_{O_R} \partial_r \varphi_R \partial_r \mathbf{b} \cdot \mathbf{b} \dif x + \delta \int_{O_R} ((\nabla \times \mathbf{b}) \times \mathbf{b}) \cdot (\nabla \varphi_R \times \mathbf{b}) \dif x \notag \\
    &=: I_1+I_2+I_3+I_4+I_5+I_6, \label{eq: |Du+Db|L2}
\end{align}
where $p_R=\frac{1}{|O_R|} \int_{O_R} p \dif x$ denotes the mean value of $p$ over $O_R$. The reason to discuss $p-p_R$ rather than $p$ is that $\nabla p(x',x_3+1)-\nabla p(x',x_3)=0$ according to \eqref{eq: MHD_Hall-MHD}$_1$ and the periodicity of $\mathbf{u},\mathbf{b}$ along the $x_3$--axis, which implies that $p(x',x_3+1)-p(x',x_3)=c$, where $c$ is a constant. Integrate it over $O_R$ with respect to $(x',x_3)$, we obtain $\int_1^2 \int_{B_R \backslash B_{R-1}} p(x',x_3+1) \dif x' \mathrm{d} (x_3+1)-\int_0^1 \int_{B_R \backslash B_{R-1}} p(x',x_3) \dif x' \dif x_3=c |O_R|$. Therefore $p-p_R$ is periodic with respect to $x_3$, and thus the boundary term vanishes when integrating the term involving $p-p_R$ by parts.
\subsection{Proof for Case (i) of Theorem \ref{thm: Liouville}}
\label{subsec: case1}
According to Lemma \ref{thm: Bogovskii} and \eqref{eq: urbrmean=0}, for every fixed $\theta \in [0,2\pi]$, there exists $\Psi_{R,\theta} \in H_0^1(\widetilde{D}_R;\R^2)$, such that
\begin{equation}
    \label{eq: divV=rur}
    \partial_r \Psi_{R,\theta}^r+\partial_z \Psi_{R,\theta}^z = r u^r \ \text{in} \ \widetilde{D}_R, \quad \text{and} \ \norm{\partial_r \Psi_{R,\theta}}{L^2(\widetilde{D}_R)} + \norm{\partial_z \Psi_{R,\theta}}{L^2(\widetilde{D}_R)} \leq C \norm{r u^r}{L^2(\widetilde{D}_R)},
\end{equation}
where $C$ is independent of $R$ and $\theta$. On the other hand, differentiating \eqref{eq: urbrmean=0} and \eqref{eq: divV=rur} with respect to $\theta$ leads to $\int_0^1 \partial_\theta u^r \dif z=0$ and
\begin{equation*}
    \partial_r \partial_\theta \Psi_{R,\theta}^r+\partial_z \partial_\theta \Psi_{R,\theta}^z=r\partial_\theta u^r \quad \text{in} \ \widetilde{D}_R,
\end{equation*}
respectively, which implies that $\partial_\theta \Psi_{R,\theta}$ is the image of $r\partial_\theta u^r$ under the Bogovski\u{i} map, and therefore
\begin{equation}
    \label{eq: divV'=rur'}
    \norm{\partial_r \partial_\theta \Psi_{R,\theta}}{L^2(\widetilde{D}_R)} + \norm{\partial_z \partial_\theta \Psi_{R,\theta}}{L^2(\widetilde{D}_R)} \leq C \norm{r \partial_\theta u^r}{L^2(\widetilde{D}_R)},
\end{equation}
where $C$ is independent of $R,\theta$. \eqref{eq: divV=rur}$_2$, \eqref{eq: divV'=rur'} combined with \eqref{eq: Poincare} then imply that
\begin{align}
    &\norm{\partial_r \Psi_{R,\theta}}{L^2(\D_R)} + \norm{\partial_z \Psi_{R,\theta}}{L^2(\D_R)} \notag \\
    \leq& C \norm{r u^r}{L^2(\D_R)} \leq C R^{1/2} \norm{u^r}{L^2(O_R)} \leq C R^{1/2} \norm{\nabla \mathbf{u}}{L^2(O_R)}, \label{eq: DV_estimate}
\end{align}
and
\begin{align}
    &\norm{\partial_r \partial_\theta \Psi_{R,\theta}}{L^2(\D_R)} + \norm{\partial_z \partial_\theta \Psi_{R,\theta}}{L^2(\D_R)} \notag \\
    \leq& C \norm{r \partial_\theta u^r}{L^2(\D_R)} \leq C R^{3/2} \norm{r^{-1} \partial_\theta u^r}{L^2(O_R)} \leq C R^{3/2} \norm{\nabla \mathbf{u}}{L^2(O_R)}. \label{eq: DV'_estimate}
\end{align}

According to \eqref{eq: divV=rur}$_1$ and the periodicity of $p-p_R$ along the $z$--axis, we have
\begin{align}
    I_3 =& \int_{O_R} (p-p_R) u^r \partial_r \varphi_R \dif x = -\int_0^1 \int_0^{2\pi} \int_{R-1}^R (p-p_R) r u^r \dif r \dif \theta \dif z \notag \\
    =& -\int_0^1 \int_0^{2\pi} \int_{R-1}^R (p-p_R) (\partial_r \Psi_{R,\theta}^r+\partial_z \Psi_{R,\theta}^z) \dif r \dif \theta \dif z \notag \\
    =& \int_0^1 \int_0^{2\pi} \int_{R-1}^R (\partial_r p \Psi_{R,\theta}^r+\partial_z p \Psi_{R,\theta}^z) \dif r \dif \theta \dif z =: I_{31}+I_{32}. \label{eq: I3}
\end{align}
The right hand side of \eqref{eq: I3} can be replaced by the terms involving $\mathbf{u},\mathbf{b}$ through \eqref{eq: cylindrical}:
\begin{align}
    I_{31} =& -\int_0^1 \int_0^{2\pi} \int_{R-1}^R \big( \partial_r u^r \partial_r \Psi_{R,\theta}^r + \partial_z u^r \partial_z \Psi_{R,\theta}^r \big) \dif r \dif \theta \dif z \notag \\
    &+ \int_0^1 \int_0^{2\pi} \int_{R-1}^R \frac{1}{r^2} \partial_\theta u^r \partial_\theta \Psi_{R,\theta}^r \dif r \dif \theta \dif z \notag \\
    &+ \int_0^1 \int_0^{2\pi} \int_{R-1}^R \frac{1}{r} \partial_r u^r \Psi_{R,\theta}^r \dif r \dif \theta \dif z - \int_0^1 \int_0^{2\pi} \int_{R-1}^R \frac{1}{r^2} u^r \Psi_{R,\theta}^r \dif r \dif \theta \dif z \notag \\
    &+ \int_0^1 \int_0^{2\pi} \int_{R-1}^R \frac{(u^\theta)^2-(b^\theta)^2}{r} \Psi_{R,\theta}^r \dif r \dif \theta \dif z \notag \\
    &- \int_0^1 \int_0^{2\pi} \int_{R-1}^R \bigg( u^r \partial_r + \frac{u^\theta}{r} \partial_\theta + u^z \partial_z \bigg) u^r \Psi_{R,\theta}^r \dif r \dif \theta \dif z \notag \\
    &+ \int_0^1 \int_0^{2\pi} \int_{R-1}^R \bigg( b^r \partial_r + \frac{b^\theta}{r} \partial_\theta + b^z \partial_z \bigg) b^r \Psi_{R,\theta}^r \dif r \dif \theta \dif z =: \sum_{i=1}^7 I_{31i}, \label{eq: dp/drVr} \\
    I_{32} =& -\int_0^1 \int_0^{2\pi} \int_{R-1}^R \big( \partial_r u^z \partial_r \Psi_{R,\theta}^z + \partial_z u^z \partial_z \Psi_{R,\theta}^z \big) \dif r \dif \theta \dif z \notag \\
    &+ \int_0^1 \int_0^{2\pi} \int_{R-1}^R \frac{1}{r^2} \partial_\theta u^z \partial_\theta \Psi_{R,\theta}^z \dif r \dif \theta \dif z + \int_0^1 \int_0^{2\pi} \int_{R-1}^R \frac{1}{r} \partial_r u^z \Psi_{R,\theta}^z \dif r \dif \theta \dif z \notag \\
    &- \int_0^1 \int_0^{2\pi} \int_{R-1}^R \bigg( u^r \partial_r + \frac{u^\theta}{r} \partial_\theta + u^z \partial_z \bigg) u^z \Psi_{R,\theta}^z \dif r \dif \theta \dif z \notag \\
    &+ \int_0^1 \int_0^{2\pi} \int_{R-1}^R \bigg( b^r \partial_r + \frac{b^\theta}{r} \partial_\theta + b^z \partial_z \bigg) b^z \Psi_{R,\theta}^z \dif r \dif \theta \dif z =: \sum_{i=1}^5 I_{32i}. \label{eq: dp/dzVz}
\end{align}

We now estimate $I_{311}$--$I_{317}$ and $I_{321}$--$I_{325}$ carefully. It follows from H\"older's inequality and \eqref{eq: DV_estimate} that
\begin{align}
    |I_{311}| \leq& C R^{-1/2} \norm{\nabla \mathbf{u}}{L^2(O_R)} \cdot R^{1/2} \norm{u^r}{L^2(O_R)} \notag \\
    \leq& C R^{1/2} \norm{\nabla \mathbf{u}}{L^2(O_R)} \norm{u^r}{L^\infty(O_R)} \leq C R^{1/2} \norm{\nabla \mathbf{u}}{L^2(O_R)}, \label{eq: dp/drVr_1-1}
\end{align}
and it holds by H\"older's inequality, Poincar\'e's inequality and \eqref{eq: DV'_estimate} that
\begin{align}
    |I_{312}| \leq& C R^{-3/2} \norm{r^{-1} \partial_\theta u^r}{L^2(O_R)} \norm{\partial_\theta \Psi_{R,\theta}^r}{L^2(\D_R)} \notag \\
    \leq& C R^{-3/2} \norm{\nabla \mathbf{u}}{L^2(O_R)} \norm{\partial_z \partial_\theta \Psi_{R,\theta}^r}{L^2(\D_R)} \notag \\
    \leq& C R^{-3/2} \norm{\nabla \mathbf{u}}{L^2(O_R)} \cdot R^{3/2} \norm{\nabla \mathbf{u}}{L^2(O_R)} = C \norm{\nabla \mathbf{u}}{L^2(O_R)}^2. \label{eq: dp/drVr_1-2}
\end{align}
Similarly, from \eqref{eq: DV_estimate} and \eqref{eq: Poincare} one can show that
\begin{align}
    |I_{313}| \leq& C R^{-3/2} \norm{\partial_r u^r}{L^2(O_R)} \norm{\Psi_{R,\theta}^r}{L^2(\D_R)} \notag \\
    \leq& C R^{-3/2} \norm{\nabla \mathbf{u}}{L^2(O_R)} \norm{\partial_z \Psi_{R,\theta}^r}{L^2(\D_R)} \notag \\
    \leq& C R^{-3/2} \norm{\nabla \mathbf{u}}{L^2(O_R)} \cdot R^{1/2} \norm{u^r}{L^2(O_R)} \notag \\
    \leq& C R^{-3/2} \norm{\nabla \mathbf{u}}{L^2(O_R)} \cdot R \norm{u^r}{L^\infty(O_R)} \leq C R^{-1/2} \norm{\nabla \mathbf{u}}{L^2(O_R)}, \label{eq: dp/drVr_2-1} \\
    |I_{314}| \leq& C R^{-5/2} \norm{u^r}{L^2(O_R)} \norm{\Psi_{R,\theta}^r}{L^2(\D_R)} \notag \\
    \leq& C R^{-5/2} \norm{\partial_z u^r}{L^2(O_R)} \norm{\partial_z \Psi_{R,\theta}^r}{L^2(\D_R)} \notag \\
    \leq& C R^{-5/2} \norm{\partial_z u^r}{L^2(O_R)} \cdot R \norm{u^r}{L^\infty(O_R)} \leq C R^{-3/2} \norm{\nabla \mathbf{u}}{L^2(O_R)}, \label{eq: dp/drVr_2-2} \\
    |I_{315}| \leq& C R^{-1} \big\|(u^\theta,b^\theta)\big\|_{L^\infty(O_R)}^2 \cdot R^{1/2} \norm{\Psi_{R,\theta}^r}{L^2(\D_R)} \notag \\
    \leq& C R^{-1/2} \big\|(u^\theta,b^\theta)\big\|_{L^\infty(O_R)}^2 \norm{\partial_z \Psi_{R,\theta}^r}{L^2(\D_R)} \notag \\
    \leq& C R^{-1/2} \big\|(u^\theta,b^\theta)\big\|_{L^\infty(O_R)}^2 \cdot R^{1/2} \norm{\nabla \mathbf{u}}{L^2(O_R)} \leq C \norm{\nabla \mathbf{u}}{L^2(O_R)}, \label{eq: dp/drVr_2-3}
\end{align}
and
\begin{align}
    |I_{316}| \leq& C \norm{\mathbf{u}}{L^\infty(O_R)} \cdot R^{-1/2} \norm{\nabla \mathbf{u}}{L^2(O_R)} \norm{\Psi_{R,\theta}^r}{L^2(\D_R)} \notag \\
    \leq& C \norm{\mathbf{u}}{L^\infty(O_R)} \cdot R^{-1/2} \norm{\nabla \mathbf{u}}{L^2(O_R)} \norm{\partial_z \Psi_{R,\theta}^r}{L^2(\D_R)} \notag \\
    \leq& C \norm{\mathbf{u}}{L^\infty(O_R)} \cdot R^{-1/2} \norm{\nabla \mathbf{u}}{L^2(O_R)} \cdot R \norm{u^r}{L^\infty(O_R)} \leq C R^{1/2} \norm{\nabla \mathbf{u}}{L^2(O_R)}, \label{eq: dp/drVr_3}
\end{align}
and likewise,
\begin{equation}
    \label{eq: dp/drVr_4}
    |I_{317}| \leq C R^{1/2} \norm{\nabla \mathbf{b}}{L^2(O_R)}.
\end{equation}
Combine \eqref{eq: dp/drVr} and \eqref{eq: dp/drVr_1-1}--\eqref{eq: dp/drVr_4}, we conclude
\begin{equation}
    \label{eq: dp/drVr_estimate}
    |I_{31}| \leq C R^{1/2} \norm{(\nabla \mathbf{u},\nabla \mathbf{b})}{L^2(O_R)} + C \norm{\nabla \mathbf{u}}{L^2(O_R)}^2.
\end{equation}

Analogously, it can be shown that
\begin{gather*}
    |I_{321}|,|I_{324}| \leq C R^{1/2} \norm{\nabla \mathbf{u}}{L^2(O_R)}, \quad |I_{322}| \leq C \norm{\nabla \mathbf{u}}{L^2(O_R)}^2, \\
    |I_{323}| \leq C R^{-1/2} \norm{\nabla \mathbf{u}}{L^2(O_R)}, \quad |I_{325}| \leq C R^{1/2} \norm{\nabla \mathbf{b}}{L^2(O_R)},
\end{gather*}
which, combined with \eqref{eq: dp/dzVz}, leads to
\begin{equation}
    \label{eq: dp/dzVz_estimate}
    |I_{32}| \leq C R^{1/2} \norm{(\nabla \mathbf{u},\nabla \mathbf{b})}{L^2(O_R)} + C \norm{\nabla \mathbf{u}}{L^2(O_R)}^2.
\end{equation}

Putting \eqref{eq: dp/drVr_estimate}, \eqref{eq: dp/dzVz_estimate} into \eqref{eq: I3} one obtains
\begin{equation}
    \label{eq: I3_estimate}
    |I_3| \leq C R^{1/2} \norm{(\nabla \mathbf{u},\nabla \mathbf{b})}{L^2(O_R)} + C \norm{\nabla \mathbf{u}}{L^2(O_R)}^2.
\end{equation}

Furthermore, applying H\"older's inequality one can deduce that
\begin{align}
    |I_1| \leq& C \norm{\nabla \mathbf{u}}{L^2(O_R)} \norm{\mathbf{u}}{L^2(O_R)} \notag \\
    \leq& C \norm{\nabla \mathbf{u}}{L^2(O_R)} \cdot R^{1/2} \norm{\mathbf{u}}{L^\infty(O_R)} \leq C R^{1/2} \norm{\nabla \mathbf{u}}{L^2(O_R)}, \label{eq: I1_estimate} \\
    |I_5| \leq& C \norm{\mathbf{b}}{L^\infty(O_R)} \cdot R^{1/2} \norm{\partial_r \mathbf{b}}{L^2(O_R)} \leq C R^{1/2} \norm{\nabla \mathbf{b}}{L^2(O_R)}, \label{eq: I5_estimate} \\
    |I_6| \leq& C \norm{\mathbf{b}}{L^\infty(O_R)}^2 \norm{\nabla \mathbf{b}}{L^1(O_R)} \leq C R^{1/2} \norm{\nabla \mathbf{b}}{L^2(O_R)}, \label{eq: I6_estimate}
\end{align}
and by \eqref{eq: Poincare} one has
\begin{align}
    |I_2| \leq& C \norm{(\mathbf{u},\mathbf{b})}{L^\infty(O_R)}^2 \cdot R^{1/2} \norm{u^r}{L^2(O_R)} \leq C R^{1/2} \norm{\nabla \mathbf{u}}{L^2(O_R)}, \label{eq: I2_estimate} \\
    |I_4| \leq& C \norm{\mathbf{u}}{L^\infty(O_R)} \norm{\mathbf{b}}{L^\infty(O_R)} \cdot R^{1/2} \norm{b^r}{L^2(O_R)} \leq C R^{1/2} \norm{\nabla \mathbf{b}}{L^2(O_R)}. \label{eq: I4_estimate}
\end{align}

Substitute \eqref{eq: I3_estimate}--\eqref{eq: I4_estimate} into \eqref{eq: |Du+Db|L2}, and we obtain the following differential inequality:
\begin{equation}
    \label{eq: differential_inequality1}
    \int_{\R^2 \times \T} (|\nabla \mathbf{u}|^2+|\nabla \mathbf{b}|^2) \varphi_R \dif x \leq C R^{1/2} \norm{(\nabla \mathbf{u},\nabla \mathbf{b})}{L^2(O_R)} + C \norm{\nabla \mathbf{u}}{L^2(O_R)}^2.
\end{equation}
Denote
\begin{equation}
    \label{eq: Y(R)}
    Y(R) = \int_{\R^2 \times \T} (|\nabla \mathbf{u}|^2+|\nabla \mathbf{b}|^2) \varphi_R \dif x,
\end{equation}
then a direct computation shows that
\begin{equation*}
    \begin{split}
    Y'(R) =& \int_0^1 \int_0^{2\pi} \bigg( \frac{\mathrm{d}}{\mathrm{d} R} \int_0^{R-1} (|\nabla \mathbf{u}|^2+|\nabla \mathbf{b}|^2) r \dif r + \frac{\mathrm{d}}{\mathrm{d} R} \int_{R-1}^R (|\nabla \mathbf{u}|^2+|\nabla \mathbf{b}|^2) (R-r) r \dif r \bigg) \dif \theta \dif z \\
    =& \int_0^1 \int_0^{2\pi} \bigg( (|\nabla \mathbf{u}(R-1)|^2+|\nabla \mathbf{b}(R-1)|^2) (R-1) \\
    &+ \int_{R-1}^R (|\nabla \mathbf{u}|^2+|\nabla \mathbf{b}|^2) r \dif r - (|\nabla \mathbf{u}(R-1)|^2+|\nabla \mathbf{b}(R-1)|^2) (R-1) \bigg) \dif \theta \dif z \\
    =& \int_0^1 \int_0^{2\pi} \int_{R-1}^R (|\nabla \mathbf{u}|^2+|\nabla \mathbf{b}|^2) r \dif r \dif \theta \dif z = \int_{O_R} (|\nabla \mathbf{u}|^2+|\nabla \mathbf{b}|^2) \dif x.
    \end{split}
\end{equation*}
Therefore \eqref{eq: differential_inequality1} reads
\begin{equation}
    \label{eq: differential_inequality}
    Y(R) \leq C_1 Y'(R) + C_2 R^{1/2} [Y'(R)]^{1/2}.
\end{equation}

If we further assume $\nabla \mathbf{u} \in L^2(\R^2 \times \T)$, then \eqref{eq: differential_inequality1} is reduced to
\begin{equation}
    \label{eq: differential_inequality4}
    Y(R) \leq C R^{1/2} [Y'(R)]^{1/2}.
\end{equation}
Assume that $\nabla \mathbf{u}$ or $\nabla \mathbf{b}$ is not identically zero, then $Y(R)>0$ for $R$ large enough, therefore the above inequality gives
\begin{equation*}
    \frac{1}{CR} \leq \bigg( -\frac{1}{Y(R)} \bigg)'.
\end{equation*}
For $R_0$ large enough and any $R>R_0$, integrate it over $[R_0,R]$, and we obtain
\begin{equation*}
    \frac{1}{C} \ln \frac{R}{R_0} \leq - \frac{1}{Y(R)}+\frac{1}{Y(R_0)} \leq \frac{1}{Y(R_0)},
\end{equation*}
which is obviously wrong for $R$ sufficiently large. Hence $\nabla \mathbf{u}=\nabla \mathbf{b} \equiv 0$. In this way we have proved that if $\nabla \mathbf{u}$ or $\nabla \mathbf{b}$ is not identically equal to zero, then $\lim_{R \to \infty} Y(R)=\infty$.

We shall now finish the proof for Case (i) using this conclusion. According to \eqref{eq: differential_inequality} we have
\begin{align}
    [Y'(R)]^{1/2} \geq& \frac{-C_2 R^{1/2}+\sqrt{C_2^2 R+4 C_1 Y(R)}}{2 C_1} \notag \\
    =& \frac{2Y(R)}{\sqrt{C_2^2 R+4C_1 Y(R)}+C_2 R^{1/2}} \geq \frac{Y(R)}{\sqrt{C_2^2 R+4C_1 Y(R)}}, \label{eq: differential_inequality2}
\end{align}
where the last inequality holds since $C_2 R^{1/2} \leq \sqrt{C_2^2 R+4C_1 Y(R)}$, which implies that
\begin{equation*}
    \frac{2}{\sqrt{C_2^2 R+4C_1 Y(R)}+C_2 R^{1/2}} \geq \frac{1}{\sqrt{C_2^2 R+4C_1 Y(R)}}.
\end{equation*}
Suppose $\nabla \mathbf{u}$ or $\nabla \mathbf{b}$ is not identically zero, then for $R$ large enough, we have $Y(R)>0$, therefore it follows from \eqref{eq: differential_inequality2} that
\begin{equation}
    \label{eq: differential_inequality3}
    [C_2^2 R Y^{-2}(R)+4C_1 Y^{-1}(R)] Y'(R) \geq 1.
\end{equation}
Set $M \geq 4C_2^2$. As we have proved that $\lim_{R \to \infty} Y(R)=\infty$, there exists $R_0$ large enough such that $Y(R_0) \geq M$. For any $R>R_0$, integrating \eqref{eq: differential_inequality3} over $[R,2R]$, we obtain
\begin{equation*}
    2R C_2^2 \bigg[ \frac{1}{Y(R)}-\frac{1}{Y(2R)} \bigg] + 4C_1 \ln \frac{Y(2R)}{Y(R)} \geq R,
\end{equation*}
therefore by noting that $1/Y(R)-1/Y(2R) \leq 1/Y(R) \leq 1/M \leq 1/(4C_2^2)$, we arrive at
\begin{equation*}
    Y(2R)/Y(R) \geq \exp \big( R/(8C_1) \big), \quad \forall R>R_0.
\end{equation*}
In particular, setting $R=8C_1 R_0,16C_1 R_0,\dots$, one can show by induction that
\begin{equation*}
    Y(2^{k+3} C_1 R_0) \geq M \exp \bigg[ \bigg( \sum_{i=1}^k 2^{k-1} \bigg) R_0 \bigg] = M \exp \big[ \big( 2^k-1 \big) R_0 \big], \quad k=1,2,\dots,
\end{equation*}
which suggests that there exists a sequence of $R_k=2^{k+3} C_1 R_0$ $(k=1,2,\dots)$ such that $Y(R_k)$ grows at least exponentially, which contradicts with \eqref{eq: |Du|2+|Db|2<=Ralpha}. Hence $\nabla \mathbf{u}=\nabla \mathbf{b} \equiv 0$, which together with the condition that $u^\theta$ and $b^\theta$ are axisymmetric implies that $\mathbf{u}=(0,0,c_1)$ and $\mathbf{b}=(0,0,c_2)$.

\subsection{Proof for Case (ii) of Theorem \ref{thm: Liouville}}
\label{subsec: case2}
According to Lemma \ref{thm: Bogovskii} and \eqref{eq: urbrmean=0}, there exists $\Psi_{R} \in H_0^1(\widetilde{D}_R;\R^2)$, such that
\begin{equation}
    \label{eq: divV=rur_case2}
    \partial_r \Psi_R^r+\partial_z \Psi_R^z = r u^r \ \text{in} \ \widetilde{D}_R, \quad \text{and} \ \norm{\partial_r \Psi_R}{L^2(\widetilde{D}_R)}+\norm{\partial_z \Psi_R}{L^2(\widetilde{D}_R)} \leq C \norm{r u^r}{L^2(\widetilde{D}_R)},
\end{equation}
where $C$ is independent of $R$ and $\theta$. \eqref{eq: divV=rur_case2}$_2$ combined with \eqref{eq: Poincare} implies that
\begin{align}
    &\norm{\partial_r \Psi_R}{L^2(\D_R)} + \norm{\partial_z \Psi_R}{L^2(\D_R)} \notag \\
    \leq& C \norm{r u^r}{L^2(\D_R)} \leq C R^{1/2} \norm{u^r}{L^2(O_R)} \leq CR^{1/2} \norm{\nabla \mathbf{u}}{L^2(O_R)}, \label{eq: DV_estimate_case2}
\end{align}

By \eqref{eq: divV=rur_case2}$_1$ and the periodicity of $p-p_R$ along the $z$--axis, we have
\begin{align}
    I_3 =& \int_{O_R} (p-p_R) u^r \partial_r \varphi_R \dif x = -\int_0^1 \int_0^{2\pi} \int_{R-1}^R (p-p_R) r u^r \dif r \dif \theta \dif z \notag \\
    =& -\int_0^1 \int_0^{2\pi} \int_{R-1}^R (p-p_R) (\partial_r \Psi_R^r+\partial_z \Psi_R^z) \dif r \dif \theta \dif z \notag \\
    =& \int_0^1 \int_0^{2\pi} \int_{R-1}^R (\partial_r p \Psi_R^r+\partial_z p \Psi_R^z) \dif r \dif \theta \dif z =: J_{31}+J_{32}. \label{eq: I3_case2}
\end{align}
Similar to Section \ref{subsec: case1}, from \eqref{eq: cylindrical}$_1$ and \eqref{eq: cylindrical}$_3$ one has
\begin{align}
    J_{31} =& -\int_0^1 \int_0^{2\pi} \int_{R-1}^R \big( \partial_r u^r \partial_r \Psi_R^r + \partial_z u^r \partial_z \Psi_R^r \big) \dif r \dif \theta \dif z \notag \\
    &+ \int_0^1 \int_0^{2\pi} \int_{R-1}^R \bigg[ \bigg( \frac{1}{r} \partial_r-\frac{1}{r^2} \bigg) u^r + \frac{(u^\theta)^2-(b^\theta)^2}{r} -\frac{2}{r^2} \partial_\theta u^\theta \bigg] \Psi_R^r \dif r \dif \theta \dif z \notag \\
    &- \int_0^1 \int_0^{2\pi} \int_{R-1}^R \big( u^r \partial_r + u^z \partial_z \big) u^r \Psi_R^r \dif r \dif \theta \dif z \notag \\
    &+ \int_0^1 \int_0^{2\pi} \int_{R-1}^R \bigg( b^r \partial_r + \frac{b^\theta}{r} \partial_\theta + b^z \partial_z \bigg) b^r \Psi_R^r \dif r \dif \theta \dif z, \label{eq: dp/drVr_case2} \\
    J_{32} =& -\int_0^1 \int_0^{2\pi} \int_{R-1}^R \bigg( \partial_r u^z \partial_r \Psi_R^z + \partial_z u^z \partial_z \Psi_R^z + \frac{1}{r^2} \partial_\theta u^z \partial_\theta \Psi_R^z \bigg) \dif r \dif \theta \dif z \notag \\
    &- \int_0^1 \int_0^{2\pi} \int_{R-1}^R \bigg( u^r \partial_r + \frac{u^\theta}{r} \partial_\theta + u^z \partial_z - \frac{1}{r} \partial_r \bigg) u^z \Psi_R^z \dif r \dif \theta \dif z \notag \\
    &+ \int_0^1 \int_0^{2\pi} \int_{R-1}^R \bigg( b^r \partial_r + \frac{b^\theta}{r} \partial_\theta + b^z \partial_z \bigg) b^z \Psi_R^z \dif r \dif \theta \dif z. \label{eq: dp/dzVz_case2}
\end{align}
It can be shown from \eqref{eq: DV_estimate_case2} and \eqref{eq: Poincare} that
\begin{align}
    &\bigg| \int_0^1 \int_0^{2\pi} \int_{R-1}^R \frac{2}{r^2} \partial_\theta u^\theta \Psi_R^r \dif r \dif \theta \dif z \bigg| \leq C R^{-3/2} \|r^{-1} \partial_\theta u^\theta\|_{L^2(O_R)} \norm{\Psi_R^r}{L^2(\D_R)} \notag \\
    \leq& C R^{-3/2} \norm{\nabla \mathbf{u}}{L^2(O_R)} \norm{\partial_z \Psi_R^r}{L^2(\D_R)} \leq C R^{-3/2} \norm{\nabla \mathbf{u}}{L^2(O_R)} \cdot R^{1/2} \norm{u^r}{L^2(O_R)} \notag \\
    \leq& C R^{-3/2} \norm{\nabla \mathbf{u}}{L^2(O_R)} \cdot R \norm{u^r}{L^\infty(O_R)} \leq C R^{-1/2} \norm{\nabla \mathbf{u}}{L^2(O_R)}, \label{eq: dp/drVr_case2_2-3}
\end{align}
and the rest terms in the right hand sides of \eqref{eq: dp/drVr_case2} and \eqref{eq: dp/dzVz_case2} can be estimated in exactly the same way as \eqref{eq: dp/drVr} and \eqref{eq: dp/dzVz}, with $\Psi_{R,\theta}$ replaced by $\Psi_R$. Note that we do not need to estimate the terms involving $\partial_\theta u^r,\partial_\theta b^r$ and $\partial_\theta \Psi_R$, as they simply vanish. Consequently, we arrive at
\begin{equation}
    \label{eq: I3_estimate_case2}
    |I_3| \leq C R^{1/2} \norm{(\nabla \mathbf{u},\nabla \mathbf{b})}{L^2(O_R)}.
\end{equation}
It can also be easily seen that $|I_1|,|I_2|$ and $|I_4|$--$|I_6|$ share the same bounds as those of $|I_1|$, $|I_2|$, $|I_4|$--$|I_6|$ in Section \ref{subsec: case1}. Therefore the function $Y(R)$ defined by \eqref{eq: Y(R)} satisfies the inequality \eqref{eq: differential_inequality4} in this case, which implies that $\nabla \mathbf{u}=\nabla \mathbf{b} \equiv 0$ according to the proof in Section \ref{subsec: case1}, and $\mathbf{u}=(0,0,c_1)$, $\mathbf{b}=(0,0,c_2)$ as a consequence of the axis-symmetry of $u^r$ and $b^r$.

\subsection{Proof for Case (iii) of Theorem \ref{thm: Liouville}}
\label{subsec: case3}
Similar to the proof of Lemma \ref{thm: Poincare}, integrate \eqref{eq: divu=divb=0} with respect to $(r,\theta,z)$, and we obtain
\begin{equation}
    \label{eq: urbrmean=0(1)}
    \int_0^1 \int_0^{2\pi} u^r \dif \theta \dif z = \int_0^1 \int_0^{2\pi} b^r \dif \theta \dif z = 0.
\end{equation}
Therefore it can be shown from the Poincar\'e's inequality that:
\begin{equation}
    \label{eq: Poincare_case3}
    \begin{split}
    &\norm{u^r}{L^2(O_R)} \leq C R^{1/2} \norm{u^r}{L^2(\D_R)} \leq C R^{1/2} \|\overline{\nabla} u^r\|_{L^2(\D_R)} \leq C R \norm{\nabla \mathbf{u}}{L^2(O_R)}, \\
    &\text{and similarly}, \quad \norm{b^r}{L^2(O_R)} \leq C R \norm{\nabla \mathbf{b}}{L^2(O_R)}.
    \end{split}
\end{equation}
Besides, according to Lemma \ref{thm: Bogovskii} and \eqref{eq: urbrmean=0(1)}, there exists $\Psi_R \in H_0^1(\D_R;\R^3)$ such that
\begin{gather}
    \label{eq: divV=rur_case3}
    \partial_r \Psi_R^r+\partial_\theta \Psi_R^\theta+\partial_z \Psi_R^z = r u^r \quad \text{in} \ \D_R, \\
    \label{eq: DV_estimate_case3}
    \norm{\partial_r \Psi_R}{L^2(\D_R)} + \norm{\partial_\theta \Psi_R}{L^2(\D_R)} + \norm{\partial_z \Psi_R}{L^2(\D_R)} \leq C \norm{r u^r}{L^2(\D_R)} \leq C R^{1/2} \norm{u^r}{L^2(O_R)},
\end{gather}
where $C$ is independent of $R$ and $\theta$.

Again, according to \eqref{eq: divV=rur_case3} and the periodicity of $p-p_R$ along the $z$--axis, we have
\begin{align}
    I_3 =& \int_{O_R} (p-p_R) u^r \partial_r \varphi_R \dif x = -\int_0^1 \int_0^{2\pi} \int_{R-1}^R (p-p_R) ru^r \dif r \dif \theta \dif z \notag \\
    =& -\int_0^1 \int_0^{2\pi} \int_{R-1}^R (p-p_R) (\partial_r \Psi_R^r+\partial_\theta \Psi_R^\theta+\partial_z \Psi_R^z) \dif r \dif \theta \dif z \notag \\
    =& \int_0^1 \int_0^{2\pi} \int_{R-1}^R (\partial_r p \Psi_R^r+\partial_\theta p \Psi_R^\theta+\partial_z p \Psi_R^z) \dif r \dif \theta \dif z =: K_{31}+K_{32}+K_{33}. \label{eq: I3_case3}
\end{align}
From \eqref{eq: cylindrical} one has
\begin{align}
    K_{31} =& -\int_0^1 \int_0^{2\pi} \int_{R-1}^R \bigg( \partial_r u^r \partial_r \Psi_R^r + \partial_z u^r \partial_z \Psi_R^r + \frac{1}{r^2} \partial_\theta u^r \partial_\theta \Psi_R^r \bigg) \dif r \dif \theta \dif z \notag \\
    &+ \int_0^1 \int_0^{2\pi} \int_{R-1}^R \frac{1}{r} \bigg( \partial_r u^r - \frac{2}{r} \partial_\theta u^\theta \bigg) \Psi_R^r \dif r \dif \theta \dif z \notag \\
    &+ \int_0^1 \int_0^{2\pi} \int_{R-1}^R \frac{(u^\theta)^2-(b^\theta)^2}{r} \Psi_R^r \dif r \dif \theta \dif z - \int_0^1 \int_0^{2\pi} \int_{R-1}^R \frac{1}{r^2} u^r \Psi_R^r \dif r \dif \theta \dif z \notag \\
    &- \int_0^1 \int_0^{2\pi} \int_{R-1}^R \bigg( u^r \partial_r + \frac{u^\theta}{r} \partial_\theta + u^z \partial_z \bigg) u^r \Psi_R^r \dif r \dif \theta \dif z \notag \\
    &+ \int_0^1 \int_0^{2\pi} \int_{R-1}^R \bigg( b^r \partial_r + \frac{b^\theta}{r} \partial_\theta + b^z \partial_z \bigg) b^r \Psi_R^r \dif r \dif \theta \dif z =: \sum_{i=1}^6 K_{31i}, \label{eq: dp/drVr_case3} \\
    K_{32} =& -\int_0^1 \int_0^{2\pi} \int_{R-1}^R \bigg[ r (\partial_r u^\theta \partial_r \Psi_R^\theta+\partial_z u^\theta \partial_z \Psi_R^\theta) + \frac{1}{r} \partial_\theta u^\theta \partial_\theta \Psi_R^\theta \bigg] \dif r \dif \theta \dif z \notag \\
    &- \int_0^1 \int_0^{2\pi} \int_{R-1}^R \bigg( \frac{1}{r} u^\theta - \frac{2}{r} \partial_\theta u^r \bigg) \Psi_R^\theta \dif r \dif \theta \dif z - \int_0^1 \int_0^{2\pi} \int_{R-1}^R \big( u^r u^\theta - b^r b^\theta \big) \Psi_R^\theta \dif r \dif \theta \dif z \notag \\
    &- \int_0^1 \int_0^{2\pi} \int_{R-1}^R r \bigg( u_r \partial_r + \frac{1}{r} u^\theta \partial_\theta + u^z \partial_z \bigg) u^\theta \Psi_R^\theta \dif r \dif \theta \dif z \notag \\
    &+ \int_0^1 \int_0^{2\pi} \int_{R-1}^R r \bigg( b_r \partial_r + \frac{1}{r} b^\theta \partial_\theta + b^z \partial_z \bigg) b^\theta \Psi_R^\theta \dif r \dif \theta \dif z =: \sum_{i=1}^5 K_{32i}, \label{eq: dp/dthetaVtheta_case3} \\
    K_{33} =& -\int_0^1 \int_0^{2\pi} \int_{R-1}^R \bigg( \partial_r u^z \partial_r \Psi_R^z + \partial_z u^z \partial_z \Psi_R^z + \frac{1}{r^2} \partial_\theta u^z \partial_\theta \Psi_R^z \bigg) \dif r \dif \theta \dif z \notag \\
    &+ \int_0^1 \int_0^{2\pi} \int_{R-1}^R \frac{1}{r} \partial_r u^z \Psi_R^z \dif r \dif \theta \dif z - \int_0^1 \int_0^{2\pi} \int_{R-1}^R \bigg( u^r \partial_r + \frac{u^\theta}{r} \partial_\theta + u^z \partial_z \bigg) u^z \Psi_R^z \dif r \dif \theta \dif z \notag \\
    &+ \int_0^1 \int_0^{2\pi} \int_{R-1}^R \bigg( b^r \partial_r + \frac{b^\theta}{r} \partial_\theta + b^z \partial_z \bigg) b^z \Psi_R^z \dif r \dif \theta \dif z =: \sum_{i=1}^4 K_{33i}. \label{eq: dp/dzVz_case3}
\end{align}

It follows from \eqref{eq: DV_estimate_case3} that
\begin{align}
    |K_{311}| \leq& C R^{-1/2} \norm{\nabla \mathbf{u}}{L^2(O_R)} \|\overline{\nabla} \Psi_R^r\|_{L^2(\D_R)} \leq C R^{-1/2} \norm{\nabla \mathbf{u}}{L^2(O_R)} \cdot R^{1/2} \norm{u^r}{L^2(O_R)} \notag \\
    \leq& C R^{1/2} \norm{\nabla \mathbf{u}}{L^2(O_R)} \norm{u^r}{L^\infty(O_R)} \leq C R^{1/2} \norm{\nabla \mathbf{u}}{L^2(O_R)}, \label{eq: dp/drVr_case3_1}
\end{align}
and it holds by \eqref{eq: DV_estimate_case3} and \eqref{eq: Poincare_case3} that
\begin{align}
    |K_{312}| \leq& C R^{-3/2} \norm{\nabla \mathbf{u}}{L^2(O_R)} \norm{\Psi_R^r}{L^2(\D_R)} \leq C R^{-3/2} \norm{\nabla \mathbf{u}}{L^2(O_R)} \norm{\partial_z \Psi_R^r}{L^2(\D_R)} \notag \\
    \leq& C R^{-3/2} \norm{\nabla \mathbf{u}}{L^2(O_R)} \cdot R^{1/2} \norm{u^r}{L^2(O_R)} \leq C R^{-1/2} \norm{\nabla \mathbf{u}}{L^2(O_R)}, \label{eq: dp/drVr_case3_2-1} \\
    |K_{313}| \leq& C R^{-1} \big\|(u^\theta,b^\theta)\big\|_{L^\infty(O_R)}^2 \norm{\Psi_R^r}{L^2(\D_R)} \leq C R^{-1} \big\|(u^\theta,b^\theta)\big\|_{L^\infty(O_R)}^2 \norm{\partial_z \Psi_R^r}{L^2(\D_R)} \notag \\
    \leq& C R^{-1} \big\|(u^\theta,b^\theta)\big\|_{L^\infty(O_R)}^2 \cdot R^{3/2} \norm{\nabla \mathbf{u}}{L^2(O_R)} \leq C R^{1/2} \norm{\nabla \mathbf{u}}{L^2(O_R)}, \label{eq: dp/drVr_case3_2-3} \\
    |K_{314}| \leq& C R^{-5/2} \norm{u^r}{L^2(O_R)} \norm{\Psi_R^r}{L^2(\D_R)} \leq C R^{-3/2} \norm{\nabla \mathbf{u}}{L^2(O_R)} \norm{\partial_z \Psi_R^r}{L^2(\D_R)} \notag \\
    \leq& C R^{-3/2} \norm{\nabla \mathbf{u}}{L^2(O_R)} \cdot R^{1/2} \norm{u^r}{L^2(O_R)} \leq C R^{-1/2} \norm{\nabla \mathbf{u}}{L^2(O_R)}, \label{eq: dp/drVr_case3_2-2}
\end{align}
and
\begin{align}
    |K_{315}| \leq& C \norm{\mathbf{u}}{L^\infty(O_R)} \cdot R^{-1/2} \norm{\nabla \mathbf{u}}{L^2(O_R)} \norm{\Psi_R^r}{L^2(\D_R)} \notag \\
    \leq& C R^{-1/2} \norm{\mathbf{u}}{L^\infty(O_R)} \norm{\nabla \mathbf{u}}{L^2(O_R)} \norm{\partial_z \Psi_R^r}{L^2(\D_R)} \notag \\
    \leq& C R^{-1/2} \norm{\mathbf{u}}{L^\infty(O_R)} \norm{\nabla \mathbf{u}}{L^2(O_R)} \cdot R^{1/2} \norm{u^r}{L^2(O_R)} \leq C R^{1/2} \norm{\nabla \mathbf{u}}{L^2(O_R)}, \label{eq: dp/drVr_case3_3}
\end{align}
and likewise,
\begin{equation}
    \label{eq: dp/drVr_case3_4}
    |K_{316}| \leq C R^{1/2} \norm{\nabla \mathbf{b}}{L^2(O_R)}.
\end{equation}
Combining \eqref{eq: dp/drVr_case3} and \eqref{eq: dp/drVr_case3_1}--\eqref{eq: dp/drVr_case3_4}, we conclude
\begin{equation}
    \label{eq: dp/drVr_estimate_case3}
    |K_{31}| \leq C R^{1/2} \norm{(\nabla \mathbf{u},\nabla \mathbf{b})}{L^2(O_R)}.
\end{equation}

Similarly, one can derive from \eqref{eq: DV_estimate_case3} and \eqref{eq: Poincare_case3} that
\begin{align*}
    |K_{321}| \leq& C R^{1/2} \norm{\nabla \mathbf{u}}{L^2(O_R)} \| \overline{\nabla} \Psi_R^\theta \|_{L^2(\D_R)} \leq C R^{3/2} \norm{\nabla \mathbf{u}}{L^2(O_R)} \norm{u^r}{L^\infty(O_R)}, \\
    |K_{322}| \leq& C \big( R^{-1} \|u^\theta\|_{L^2(\D_R)} + \norm{r^{-1} \partial_\theta u^r}{L^2(\D_R)} \big) \|\Psi_R^\theta\|_{L^2(\D_R)} \\
    \leq& C \big( R^{-1} \|u^\theta\|_{L^\infty(O_R)} + R^{-1/2} \norm{\nabla \mathbf{u}}{L^2(O_R)} \big) \|\partial_z \Psi_R^\theta\|_{L^2(\D_R)} \leq CR^{1/2} \norm{\nabla \mathbf{u}}{L^2(O_R)}, \\
    |K_{323}| \leq& C \norm{(u^r,b^r)}{L^\infty(\D_R)} \|(u^\theta,b^\theta)\|_{L^\infty(\D_R)} \|\Psi_R^\theta\|_{L^2(\D_R)} \\
    \leq& C \norm{(u^r,b^r)}{L^\infty(\D_R)} \|(u^\theta,b^\theta)\|_{L^\infty(\D_R)} \|\partial_z \Psi_R^\theta\|_{L^2(\D_R)} \\
    \leq& C \norm{(u^r,b^r)}{L^\infty(\D_R)} \|(u^\theta,b^\theta)\|_{L^\infty(\D_R)} \cdot R^{1/2} \norm{u^r}{L^2(O_R)} \\
    \leq& C R^{3/2} \norm{(u^r,b^r)}{L^\infty(O_R)} \norm{\nabla \mathbf{u}}{L^2(O_R)}, \\
    |K_{324}| \leq& C R \norm{\mathbf{u}}{L^\infty(\D_R)} \cdot R^{-1/2} \norm{\nabla \mathbf{u}}{L^2(O_R)} \| \Psi_R^\theta \|_{L^2(\D_R)} \leq C R^{3/2} \norm{\nabla \mathbf{u}}{L^2(O_R)} \norm{u^r}{L^\infty(O_R)}, \\
    |K_{325}| \leq& C R \norm{\mathbf{b}}{L^\infty(\D_R)} \cdot R^{-1/2} \norm{\nabla \mathbf{b}}{L^2(O_R)} \| \Psi_R^\theta \|_{L^2(\D_R)} \leq C R^{3/2} \norm{\nabla \mathbf{b}}{L^2(O_R)} \norm{u^r}{L^\infty(O_R)},
\end{align*}
which combined with \eqref{eq: dp/dthetaVtheta_case3} shows that
\begin{equation}
    \label{eq: dp/dthetaVtheta_estimate_case3}
    |K_{32}| \leq C R^{1/2} \norm{(\nabla \mathbf{u},\nabla \mathbf{b})}{L^2(O_R)} \big( 1 + R \norm{(u^r,b^r)}{L^\infty(O_R)} \big).
\end{equation}

Also, it holds that
\begin{gather*}
    |K_{331}| \leq C R^{1/2} \norm{\nabla \mathbf{u}}{L^2(O_R)} \norm{u^r}{L^\infty(O_R)}, \quad |K_{332}| \leq C R^{-1/2} \norm{\nabla \mathbf{u}}{L^2(O_R)} \norm{u^r}{L^\infty(O_R)}, \\
    |K_{333}| \leq C R^{1/2} \norm{\nabla \mathbf{u}}{L^2(O_R)} \norm{u^r}{L^\infty(O_R)}, \quad |K_{334}| \leq C R^{1/2} \norm{\nabla \mathbf{b}}{L^2(O_R)} \norm{u^r}{L^\infty(O_R)},
\end{gather*}
which, combined with \eqref{eq: dp/dzVz_case3}, leads to
\begin{equation}
    \label{eq: dp/dzVz_estimate_case3}
    |K_{33}| \leq C R^{1/2} \norm{(\nabla \mathbf{u},\nabla \mathbf{b})}{L^2(O_R)}.
\end{equation}

Putting \eqref{eq: dp/drVr_estimate_case3}--\eqref{eq: dp/dzVz_estimate_case3} into \eqref{eq: I3_case3} one obtains
\begin{equation}
    \label{eq: I3_estimate_case3}
    |I_3| \leq C R^{1/2} \norm{(\nabla \mathbf{u},\nabla \mathbf{b})}{L^2(O_R)} \big( 1 + R \norm{(u^r,b^r)}{L^\infty(O_R)} \big).
\end{equation}

Furthermore, the terms $I_1,I_5$ and $I_6$ can be estimated in exact the same way as that in Section \ref{subsec: case1}, and by H\"older's inequality one has
\begin{gather}
    \label{eq: I2_estimate_case3}
    |I_2| \leq C \norm{(\mathbf{u},\mathbf{b})}{L^\infty(O_R)}^2 \norm{u^r}{L^1(O_R)} \leq C R \norm{u^r}{L^\infty(O_R)}, \\
    \label{eq: I4_estimate_case3}
    |I_4| \leq C \norm{\mathbf{u}}{L^\infty(O_R)} \norm{\mathbf{b}}{L^\infty(O_R)} \norm{b^r}{L^1(O_R)} \leq C R \norm{b^r}{L^\infty(O_R)}.
\end{gather}

Substitute \eqref{eq: I3_estimate_case3}--\eqref{eq: I4_estimate_case3} and \eqref{eq: I1_estimate}--\eqref{eq: I6_estimate} into \eqref{eq: |Du+Db|L2}, and we obtain
\begin{equation}
    \label{eq: differential_inequality_case3}
    \int_{\R^2 \times \T} (|\nabla \mathbf{u}|^2+|\nabla \mathbf{b}|^2) \varphi_R \dif x \leq C R^{1/2} \norm{(\nabla \mathbf{u},\nabla \mathbf{b})}{L^2(O_R)} \big( 1 + R \norm{(u^r,b^r)}{L^\infty(O_R)} \big).
\end{equation}
Noting that $(R-1) \norm{(u^r,b^r)}{L^\infty(O_R)} \leq \norm{(ru^r,rb^r)}{L^\infty(O_R)} \leq \norm{(ru^r,rb^r)}{L^\infty(\R^2 \times \T)}<\infty$, \eqref{eq: differential_inequality_case3} is reduced to
\begin{equation*}
    \int_{\R^2 \times \T} (|\nabla \mathbf{u}|^2+|\nabla \mathbf{b}|^2) \varphi_R \dif x \leq C R^{1/2} \norm{(\nabla \mathbf{u},\nabla \mathbf{b})}{L^2(O_R)},
\end{equation*}
which is exactly \eqref{eq: differential_inequality4} and implies that $\nabla \mathbf{u}=\nabla \mathbf{b} \equiv 0$. Therefore, $\mathbf{u}$ and $\mathbf{b}$ are constants.

The proof of Theorem \ref{thm: Liouville} is complete.

\vskip 3mm

\noindent{\bf Acknowledgment.} {This work was supported by NSFC grants 12070144 and 12271423.}

\end{document}